\documentclass[pdflatex,sn-mathphys-num]{sn-jnl}

\usepackage{graphicx}%
\usepackage{multirow}%
\usepackage{amsmath,amssymb,amsfonts}%
\usepackage{amsthm}%
\usepackage{mathrsfs}%
\usepackage[title]{appendix}%
\usepackage{xcolor}%
\usepackage{textcomp}%
\usepackage{manyfoot}%
\usepackage{booktabs}%
\usepackage{algorithm}%
\usepackage{quiver}
\usepackage{algorithmicx}%
\usepackage{algpseudocode}%
\usepackage{listings}%
\usepackage{enumitem}%

\theoremstyle{thmstyleone}%
\newtheorem{theorem}{Theorem}
\numberwithin{theorem}{section} 
\newtheorem{corollary}[theorem]{Corollary}

\newcommand{\Z}{\mathbb{Z}}
\newcommand{\R}{\mathbb{R}}

\theoremstyle{thmstyletwo}%

\theoremstyle{thmstylethree}%

\begin{document}

\title[Perfect functions on Stratifolds]{Perfect discrete Morse functions on Stratifolds}


\author*[1]{\fnm{José de Jesús} \sur{Liceaga Martínez}}\email{jose.liceaga@cimat.mx}

\author[1]{\fnm{Jesús} \sur{Rodríguez Viorato}}\email{jesusr@cimat.mx}

\author[1]{\fnm{José Carlos} \sur{Gomez Larrañaga}}\email{jcarlos@cimat.mx}

\equalcont{These authors contributed equally to this work.}

\affil*[1]{\orgdiv{CIMAT}, \orgname{Organization}, \orgaddress{\street{Street}, \city{City}, \postcode{100190}, \state{State}, \country{Country}}}


\nocite{*}

\abstract{In this paper, we study the computation of optimal discrete Morse functions on stratifolds. In particular, we present an algorithm that efficiently computes such functions for a broad class of them. Moreover, we characterize the conditions under which these functions are perfect.}

\keywords{topology, discrete Morse function, optimal, perfect, stratifolds}



\maketitle

\section{Introduction}\label{sec1}
Discrete Morse theory, introduced by Forman in the 1990s \cite{Forman}, provides a combinatorial counterpart to classical smooth Morse theory \cite{MorseContinuous}, in such a way that it is useful to study the topology of discrete objects like simplicial or CW complexes. One of its central features is that it allows us to compute the homology of a complex by focusing just on a distinguished subset of its cells known as \textit{critical} cells. This is particularly important in applications, as the number of operations required to compute the homology groups increases dramatically when we add more simplices \cite{AlgoritmoMatrices}. Hence, it is useful to obtain discrete Morse functions with the minimum amount of critical cells possible. Such functions are called \textit{optimal} discrete Morse functions.     

The problem of finding optimal discrete Morse functions has been widely studied, particularly for low-dimensional simplicial complexes, such as dimensions 1 and 2, as in \cite{Optimal, Perfect}. However, as Ayala et al. \cite{Perfect} noted, the existence of optimal discrete Morse functions has been mainly tackled from the point of view of \textit{perfect} discrete Morse functions \cite{Perfect, 3-mainifolds, connected-sums}, which are a more restrictive class of optimal functions. 

In this work, we study the existence of optimal discrete Morse functions for \textit{stratifolds}, a concept developed by Gómez-Larrañaga et al. \cite{estratificies}. In particular, we introduce a special type of stratifolds for which we obtain optimal discrete Morse functions in linear time. Moreover, we characterize precisely when these functions are perfect by studying their Morse matchings. 

\section{Preliminaries}

In the following section, we will introduce the basic concepts and theorems regarding discrete Morse theory. For a more in-depth review, the reader may consult \cite{UsersGuide, texbook}. 

Given a topological space \(X\) and a triangulation \(K\) of \(X\), a \textbf{discrete Morse function} (dmf) is a function \(f: K \to \mathbb{R}\) such that for each cell \(\sigma^i \in K\) of dimension \(i\) it follows that 

\begin{enumerate}
    \item \(|\{\tau \in K : \sigma^i < \tau^{i + 1}, f(\tau) \leq f(\sigma)\}| \leq 1\), and \medskip

    \item \(|\{\tau \in K : \tau^{i - 1} < \sigma^i, f(\sigma) \leq f(\tau)\}| \leq 1\),
\end{enumerate}

where \(\sigma < \tau\) means that \(\sigma\) is a face of \(\tau\). If a cell \(\sigma^i \in K\) is such that 

\begin{enumerate}
    \item \(|\{\tau \in K : \sigma^i < \tau^{i + 1}, f(\tau) \leq f(\sigma)\}| = 0\), and \medskip

    \item \(|\{\tau \in K : \tau^{i - 1} < \sigma^i, f(\sigma) \leq f(\tau)\}| = 0\),
\end{enumerate}
we say that it is a \textbf{critical cell}. The number of critical cells of dimension \(i\) that a dmf \(f\) has will be denoted as \(m_i(f)\). Moreover, the superscript of a cell will be dropped when its dimension is clear. 

\medskip 
A \textbf{discrete vector field} over a simplicial complex \(K\) is a subset \(V \subseteq \{(\tau^i, \sigma^{i + 1}) \in K \times K ~|~ \tau^i < \sigma^{i + 1}\}\) in which every cell appears in at most one pair. We can think about these vector fields geometrically as drawing an arrow from \(\tau^i\) to \(\sigma^{i + 1}\) if \((\tau^i, \sigma^{i + 1}) \in V\). Another way to visualize them is by considering their Morse matchings \cite{Matchings}: that is, the Hasse diagram of the partial order induced by inclusion on \(K\), where we highlight the pairs \((\tau^i, \sigma^{i + 1})\). These representations are shown in Figure \ref{fig:representations}. 

\bigskip 

\begin{figure}[h]
    \centering
    \includegraphics[scale = 0.8]{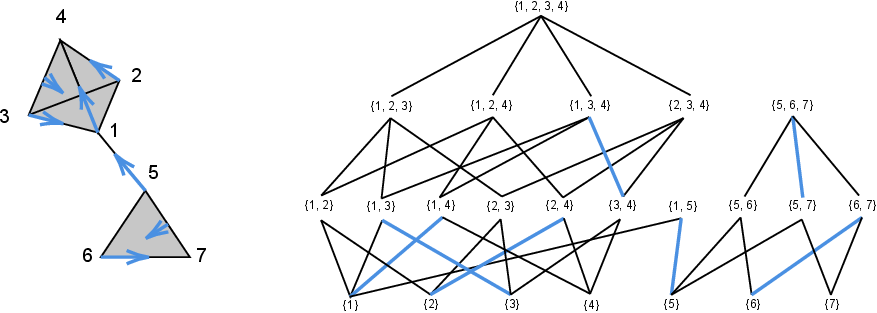}
    \caption{Left: Geometric representation of a discrete vector field \(V\). Right: Morse matching of the same discrete vector field \(V\). The pairs belonging to \(V\) are marked in blue.}
    \label{fig:representations}
\end{figure}

\medskip 
It is clear that any dmf \(f : K \to \R\) induces a discrete vector field \(-\nabla f := \{(\tau^i, \sigma^{i + 1}) ~|~ \tau^i < \sigma^{i + 1} \mbox{ and } f(\tau^i) \geq f(\sigma^{i + 1})\}\). It turns out that under certain conditions, the converse is also true.

\medskip 
Given a discrete vector field \(V\) over a simplicial complex \(K\), a \textbf{\(V\)-path} is a sequence of cells 

\[\tau_0^i < \sigma_0^{i + 1} > \tau_1^i < \sigma_1^{i + 1} > \dots > \tau_{r - 1}^i < \sigma_{r - 1}^{i + 1} > \tau_r^i\]

\medskip
where every pair \((\tau_j^i, \sigma_j^{i + 1})\) belongs to \(V\) and \(\tau_j^i \neq \tau_{j + 1}^i\). If \(r \geq 2\) and \(\tau_0^i = \tau_1^i\) we say that the path is \textbf{closed}.

\begin{theorem}\label{closed paths}
    Let \(V\) be a discrete vector field over \(K\). Then there exists a dmf \(f : K \to \R\) such that \(V = -\nabla f\) if and only if there are no closed \(V\)-paths.
\end{theorem}

A discrete vector field \(V\) satisfying \(V = - \nabla f\) for some dmf \(f\) is called a \textbf{gradient vector field}. Note that the critical cells of any given dmf \(f\) are those that do not belong to any pair in \(-\nabla f\).  

A dmf \(f : K \to \mathbb{R}\) is called \textbf{optimal} if it has fewer critical cells than any other dmf. That is, if \(\sum_im_i(f) \leq \sum_im_i(g)\) for any dmf \(g : K \to \mathbb{R}\). Finding optimal dmf's turns out to be a complex task, in which the following result has been widely used.

\begin{theorem}\label{Morse Inequalities}
    Let \(f : K \to \R\) be a dmf where \(K\) is a simplicial complex of dimension \(n\), let \(\mathbb{K}\) be a field or \(\mathbb{Z}\) and \(\beta_i(K; \mathbb{K})\) the \(i\)-th Betti number of \(K\), with \(i \in \mathbb{Z}_{\geq 0}\). Then,  

    \begin{enumerate}
        \item \(m_i(f) \geq \beta_i(K; \mathbb{K})\). \bigskip 

        \item \(m_i(f) - m_{i - 1}(f) + \dots \pm m_0(f) \geq \beta_i(K; \mathbb{K}) - \beta_{i - 1}(K; \mathbb{K}) + \dots \pm \beta_0(K; \mathbb{K})\). \bigskip

        \item \(\chi(K) = m_0(f) - m_1(f) + \dots + (-1)^nm_n(f)\).
    \end{enumerate}
\end{theorem}

A dmf is called \(\mathbb{K}\)-\textbf{perfect} if \(m_i(f) = \beta_i(K; \mathbb{K})\) for all \(i\). By Part 1 of Theorem \ref{Morse Inequalities} it is clear that if a dmf is \(\mathbb{K}\)-perfect for some \(\mathbb{K}\), then it is also optimal. However, the converse is not necessarily true. A counterexample is presented in Figure 2 of \cite{Perfect}.

As we mentioned in the Introduction, the existence of optimal dmf's has been widely studied in the literature, and there are well-established results for simplicial complexes of dimension 1 and 2. 

In the case of 1-dimensional complexes, we have a really general theorem, whose proof we present, as it will be relevant for our work. 

\begin{theorem}\label{Perfect in Graphs}
    Let \(G\) be a connected graph and \(\mathbb{K}\) a field or \(\mathbb{Z}\). Then \(G\) admits a \(\mathbb{K}\)-perfect dmf.
\end{theorem}
\begin{proof}
    Let \(n\) and \(m\) be the number of vertices and edges of \(G\), respectively, where \(m \geq n - 1\). It is a well-known fact that \(\beta_1(G; \mathbb{K}) = m - n + 1\), which is equal to the number of edges that remain after removing a spanning tree from \(G\). Moreover, \(\beta_0(G; \mathbb{K}) = 1\), since \(G\) is connected. With this in mind, we do the following:
    \begin{enumerate}
        \item Let \(T\) be a spanning tree of \(G\).
        \item Choose any vertex \(v\) of \(T\) as its root.
        \item Consider the set \(V\) of all the pairs \((w, e)\) where \(e\) is an edge of \(T\) and \(w\) is the vertex of \(e\) farthest from \(v\).
    \end{enumerate}
    Note that \(V\) is a discrete vector field, since two different edges cannot share the same vertex \(w\) as the farthest from \(v\). If this were the case, we would have multiple simple paths from \(v\) to \(w\), implying that \(T\) has a cycle, which is a contradiction.

    Moreover, by construction all edges of \(T\) belong to a pair in \(V\), and there are no closed \(V\)-paths. Hence, the dmf associated with \(V\) has \(1\) critical vertex (the root \(v\)) and \(m - n + 1\) critical edges, making it \(\mathbb{K}\)-perfect. 
\end{proof}

For 2-dimensional simplicial complexes, things get more complicated. Theorem \ref{Perfect in Graphs} can be extended to the case where we have a 2-dimensional simplicial complex that collapses to a graph, giving us the following result \cite{Perfect}. 

\begin{theorem}\label{Perfect in Collapsible for Z}
    Any 2-dimensional simplicial complex that collapses to a graph admits a \(\mathbb{Z}\)-perfect dmf. 
\end{theorem}

By the Universal Coefficient Theorem, for any field \(\mathbb{F}\) we have that \(H_i(K; \mathbb{F}) \cong H_i(K; \mathbb{Z}) \otimes \mathbb{F} \oplus \mbox{Tor}_1(H_{i - 1}(K; \mathbb{Z}), \mathbb{F})\), from where \(\beta_i(K; \mathbb{F}) \geq \beta_i(K; \mathbb{Z})\). This implies that a \(\mathbb{Z}\)-perfect dmf is also a \(\mathbb{F}\)-perfect dmf for any field. Hence, we have a more general result.

\begin{theorem}\label{Perfect in Collapsible for K}
    Any 2-dimensional simplicial complex that collapses to a graph admits a \(\mathbb{K}\)-perfect dmf, where \(\mathbb{K}\) is a field or \(\mathbb{Z}\). 
\end{theorem}

Lewiner et al. \cite{Optimal} showed that there exist optimal dmf's for any 2-manifold and gave an explicit construction of them. Moreover, Ayala et al. \cite{Perfect} proved the following theorem.

\begin{theorem}\label{Perfect in manifolds}
    Any compact and connected 2-manifold admits a \(\mathbb{Z}_2\)-perfect dmf.  
\end{theorem}

Finally, it is worth mentioning that Ayala et al. also noted in \cite{Perfect} the following result, applying to \textit{pseudo-projective} spaces, which are obtained by gluing a 2-cell to \(\mathbb{S}^1\) by a map of degree \(p\). These spaces turn out to be examples of our stratifolds. 

\begin{theorem}\label{Perfect in pseudo-projective}
    Any pseudo-projective space admits a \(\mathbb{Z}_p\)-perfect dmf for some prime number \(p\). 
\end{theorem}

\section{Stratifolds}
As defined in \cite{estratificies}, a 2-stratifold is a compact, connected Hausdorff space \(X\) together with a filtration \(\varnothing \subset X_1 \subset X_2 = X\) such that \(X_1\) is a closed (not necessarily connected) 1-manifold, each point \(x \in X_1\) has a neighborhood homeomorphic to \(\R \times CL\), where \(CL\) is the open cone of some finite set \(L\) of cardinality strictly greater than 2, and each \(x \in X_2 \setminus X_1\) has a neighborhood homeomorphic to \(\R^2\).

In other words, we can think of a 2-stratifold \(X\) as a pair \((\mathcal{M}, \mathcal{C})\), where \(\mathcal{C} := \{c_i\}_{i \in I}\) is a collection of circles, and \(\mathcal{M} = \{\mathbb{M}_j\}_{j \in J}\) is a collection of compact surfaces with boundary, such that each boundary component is attached to a unique circle via a covering map \(g : \mathbb{S}^1 \to \mathbb{S}^1\) of some number of sheets, and for each circle, the sum of the absolute values of the degrees of all the maps associated with it is strictly greater than 2. In this work, we will restrict ourselves to \textit{finite stratifolds}; this is, stratifolds such that \(|\mathcal{M}| < \infty\) and \(|\mathcal{C}| < \infty\). From now on, we will drop the \textit{finite} adjective when talking about finite stratifolds.

Now, we will present a CW structure for 2-stratifolds. By the classification theorem of surfaces \cite{Lee}, we know that every compact, connected 2-manifold without boundary is homeomorphic to one of the following:
\begin{enumerate}
    \item A sphere.
    \item A connected sum of \(g\) tori.
    \item A connected sum of \(g\) projective planes.
\end{enumerate}

In the first case, the 2-manifold is homeomorphic to a CW complex with one 0-cell, no 1-cells, and one 2-cell. In the second case, the surface admits a polygonal presentation \cite{Lee} of the form \[\langle a_1, b_1, \dots, a_g, b_g | a_1b_1a_1^{-1}b_1^{-1} \dots a_gb_ga_g^{-1}b_g^{-1} \rangle,\] 
which corresponds to a CW complex with one 0-cell, \(2g\) 1-cells and one 2-cell. Similarly, in the third case, the surface admits the following polygonal presentation \cite{Lee}: \[\langle a_1, \dots, a_g | a_1^2 \dots a_g^2 \rangle,\]
which has one 0-cell, \(g\) 1-cells, and one 2-cell. Hence, any surface can be represented as a CW complex consisting of one 0-cell, one 2-cell, and a number of \(1\)-cells that depend on the surface. 

Now, consider a surface \(\mathbb{M}_i\) with \(k_i\) boundaries. Such a surface can be obtained by taking a surface with no boundaries and removing \(k\) disjoint 2-discs. This results in a CW complex with the same number of cells as the associated surface, plus \(k_i\) additional 1-cells, denoted \(d_{i, j}\), connecting the original 0-cell \(v_i\) to newly introduced 0-cells \(v_{i, j}\), which in turn have 1-cells \(e_{i, j}\) attached to them, representing the boundary components.

Finally, when we identify the boundary components with the circles, we effectively make \(e_{i, j} = c_{f(i, j)}^{w_{i, j}}\), where \(w_{i, j}\) is the degree of the attaching map and \(f(i, j)\) is the circle to which we are attaching the \(j\)-th boundary component of \(\mathbb{M}_i\). Moreover, we may also identify some of the \(v_{i, j}\). Hence, we will exchange \(v_{i, j}\) with \(w_{f(i, j)}\).

From the former discussion, we can conclude that 2-stratifolds are homeomorphic to a CW complex that can be viewed as the union of subcomplexes of the forms shown in Figure \ref{fig:cw} (one for each surface in \(\mathcal{M}\)), where the \(c_j\)'s are the same as those from \(\mathcal{C}\). For simplicity, we omit the labels of the \(w_{f(i, j)}\) vertices. 

\begin{figure}[h]
    \centering
    \includegraphics[scale = .75]{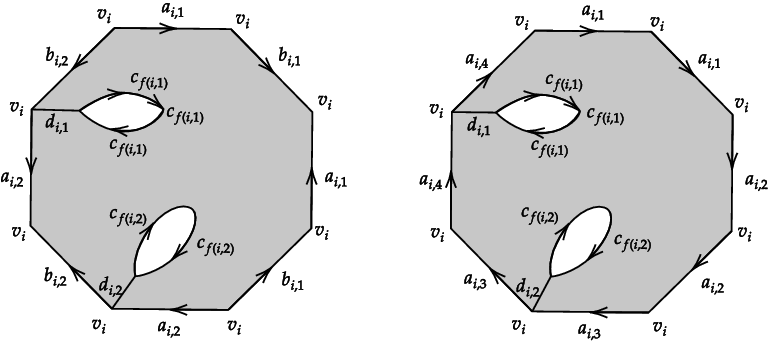}
    \caption{CW structure of the spaces in \(\mathcal{M}\) after identifications.}
    \label{fig:cw}
\end{figure}

In Figure \ref{fig:strat}, we show an example of the CW structure for a given stratifold \(X_1 = X(\mathcal{M}, \mathcal{C})\), where \(|\mathcal{C}| = 2\), while \(\mathcal{M}\) consists of two copies of \(\mathbb{T}^2 \# \mathbb{T}^2\), each one with two boundary components and a copy of \(\mathbb{R}\mathbb{P}^2\) with one boundary component.

\begin{figure}[h]
    \centering
    \includegraphics[scale = .65]{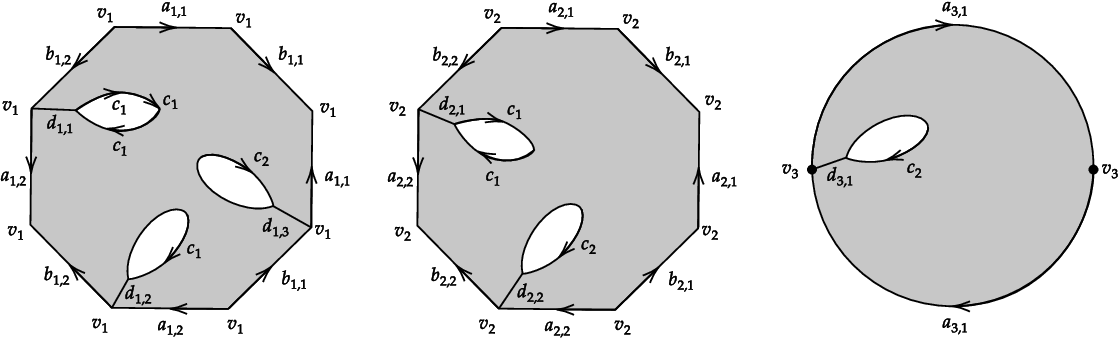}
    \caption{CW structure of a stratifold \(X_1\) with \(\mathcal{M} = \{\mathbb{T}^2 \# \mathbb{T}^2, \mathbb{T}^2 \# \mathbb{T}^2, \mathbb{R}P^2\}\).}
    \label{fig:strat}
\end{figure}

After fixing an orientation on each orientable surface \(\mathbb{M}_i \in \mathcal{M}\) and each curve  \(c_i \in \mathcal{C}\), we can associate a unique weighted bicolored multigraph \(G(X)\) to the stratifold \(X(\mathcal{M}, \mathcal{C})\) in the following way: for each \(\mathbb{M}_i \in \mathcal{M}\) we take a white vertex, and for each \(c_j \in \mathcal{C}\) we take a black vertex. Moreover, for every map that attaches \(\mathbb{M}_i\) to \(c_j\), we add an edge \(e\) between the two corresponding vertices and assign it a weight \(w(e)\) equal to the degree of the attaching map; if $\mathbb{M}_i$ is orientable, this number can be positive or negative; otherwise it is positive. We denote the set of edges that go from the vertex corresponding to \(\mathbb{M}_i\) to the vertex corresponding to \(c_j\) as \(E(i, j)\). Finally, we can label each white vertex with the genus of its corresponding surface, using negative numbers for nonorientable surfaces. This convention follows Neumann \cite{neumann}.

Note that the signs of the weights on the edges of $G(X)$ depend on the orientation given to each \(\mathbb{M}_i \in \mathcal{M}\) and each curve  \(c_i \in \mathcal{C}\). Besides those signs, everything else is determined by \(X\).

In Figure \ref{fig:graph_stratifold} we show an example of such a graph.

\begin{figure}[h]
    \centering
    \includegraphics[scale = .75]{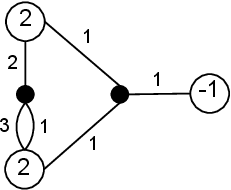}
    \caption{Graph associated with the stratifold \(X_1\) shown in \ref{fig:strat}.}
    \label{fig:graph_stratifold}
\end{figure}

\begin{theorem}\label{Homology Orientable}
    Let \(X(\mathcal{M}, \mathcal{C})\) be a 2-stratifold. If all the weights of the white vertices of \(G(X)\) are non-negative (that is, we only have oriented manifolds), and there exists a prime \(p\) such that \(p|\sum_{e \in E(i, j)}w(e)\) for all pairs \((i, j)\), then \(\beta_2(X; \Z_p) = n\), where \(n = |\mathcal{M}|\). In any other case \(\beta_2(X; \mathbb{K}) < n\), where \(\mathbb{K}\) is any field or \(\mathbb{Z}\). 
\end{theorem}

\begin{proof}
    Consider the CW decomposition shown on the left of Figure \ref{fig:cw}, assume that there exists a prime \(p\) as desired, and take the chain complex \((C_*(X; \Z_p), \partial_*)\). Note that \(C_2(X; \Z_p) = \Z_p^n\). Hence, for all generators \(F_i \in C_2(X; \Z_p)\), if we denote as \(k_i\) the number of boundary components of \(\mathbb{M}_i\), we have that
    \begin{align*}
        \partial_2(F_i) &= \sum_{j = 1}^{k_i}(d_{i, j} + w_{i, j}c_{f(i, j)} - d_{i, j}) + \sum_{k = 1}^g(a_k + b_k - a_k - b_k) \\[1ex]
        &= \sum_{j = 1}^{k_i}w_{i, j}c_{f(i, j)} = \sum_{j' = 1}^{|\mathcal{C}|}\left(\sum_{e \in E(i, j')}w(e)\right)c_{j'} = 0,
    \end{align*}

    where the last equality follows from the fact that we are working over \(\Z_p\). Hence, \(\mbox{rank}(\partial_2) = 0\) and \(\mbox{dim}(\ker(\partial_2)) = n - 0 = n\). Since \(\mbox{rank}(\partial_3) = 0\), we conclude that \(\beta_2(X, \Z_p) = n\).

    Now, if there does not exist such \(p\), then for any \(\mathbb{K}\) there is at least one \(F_i\) such that \(\partial_2(F_i) \neq 0\), implying that \(\mbox{dim}(\ker(\partial_2)) < n\), from where \(\beta_2(X, \Z_p) < n\).
\end{proof}

In a similar fashion, we have the following result for the case where we have non-orientable surfaces. Note that we need to ask for the prime to be 2, as the second sum of \(\partial_2(F_i)\) does not cancel in this case, and we do not know the signs of \(w_{i, j}c_{f(i, j)}\).

\begin{theorem}\label{Homology Non Orientable}
        Let \(X(\mathcal{M}, \mathcal{C})\) be a 2-stratifold where at least one of the weights of the white vertices of \(G(X)\) is negative, and \(2|\sum_{e \in E(i, j)}w(e)\) for all pairs \((i, j)\), then \(\beta_2(X; \Z_2) = n\), where \(n = |\mathcal{M}|\). In any other case \(\beta_2(X; \mathbb{K}) < n\), where \(\mathbb{K}\) is any field or \(\mathbb{Z}\). 
\end{theorem}

\section{Simplicial decompositions of  finite stratifolds}

In this section, we will describe the simplicial decompositions of any finite stratifold, that is, we will give a configuration from which we can recreate all possible triangulations of a finite stratifold. This will be the starting point of our main result from Sec. \ref{sec:main-result}.

The configuration that accomplishes the purpose mentioned above is that of multiple polygons $P_1, P_2, \dots, P_k$ with identified sides. A polygon $P$ with \emph{identified sides} is the simplicial complex resulting from taking the quotient of a triangulated polygon $P$ from the Euclidean space under a relation between its edges in $\partial P$.

\begin{theorem}\label{Simplicial Decomposition}
Let $K$ be any triangulation for the finite stratifold \(X(\mathcal{M}, \mathcal{C})\) such that \(|\mathcal{M}| = n\) and \(|\mathcal{C}| = r\). Then $K$ is simplicially homeomorphic to a collection of triangulated polygons $P_1, P_2, \dots, P_n$ with identified sides.
\end{theorem}

\begin{proof}
By definition, $X = X(\mathcal{M}, \mathcal{C})$ is the quotient space obtained from the disjoint union of compact surfaces $\mathbb{M}_1, \dots, \mathbb{M}_n$ by identifying each boundary component $e_{i,j} \subset \partial \mathbb{M}_i$ with a circle $c_{f(i,j)} \in \mathcal{C}$ via a rotation of angle $2\pi / w_{i,j}$.

Let
\[
\rho \colon \bigsqcup_i \mathbb{M}_i \to X
\]
denote the quotient map. Since the identifications occur only along boundary components, we have
\[
X \setminus \bigcup_{i,j} c_{f(i,j)} \;\cong\; \bigsqcup_i \bigl(\mathbb{M}_i \setminus \partial \mathbb{M}_i\bigr).
\]

By construction, $X$ is a $2$–dimensional CW complex, and therefore any triangulation $K$ of $X$ has dimension two. For each $2$–simplex $F \in K$, its interior $F^\circ$ is an open subset of $X$, so every point $p \in F^\circ$ has a neighborhood homeomorphic to $\mathbb{R}^2$.

On the other hand, as shown in \cite{estratificies}, each point of a circle $c \in \mathcal{C}$ has a neighborhood homeomorphic to $\mathbb{R} \times CN$, where $CN$ denotes the open cone on a finite set $N$ with $|N| > 2$. Such points are not locally planar, and hence cannot lie in the interior of any $2$–simplex. It follows that each circle $c \in \mathcal{C}$ is contained in the $1$–skeleton of $K$.

Since the $1$–skeleton of $K$ is a finite $1$–dimensional simplicial complex and each $c \in \mathcal{C}$ is a closed subset homeomorphic to $S^1$, it follows that every $c_{f(i,j)}$ is a 1-dimensional complex of $K$.

Now fix $i$, and let $\rho_i \colon \mathbb{M}_i \to X$ be the restriction of $\rho$. The map $\rho_i$ induces a homeomorphism from $\mathbb{M}_i \setminus \partial \mathbb{M}_i$ onto a connected component 
\[
M_i \subset X \setminus \bigcup_{j} c_{f(i,j)},
\]
and restricts to a finite covering map from each boundary component $e_{i,j} \subset \partial \mathbb{M}_i$ onto $c_{f(i,j)}$. And lifting the triangulation in $M_i$ to $\mathbb{M}_i - \partial \mathbb{M}_i$.

Thus, the triangulation $K$ lifts to a triangulation of each surface $\mathbb{M}_i$. Implyging that every triangulation $K$ is obtained by taking the quotient of a triangulated surfaces under the boundary identifications. Since every triangulation of a compact surface with boundary arises from a triangulation of a polygon with identified sides, each $\mathbb{M}_i$ is triangulated by some polygon $P_i$. Therefore, $K$ is simplicially homeomorphic to the disjoint union
\[
P_1 \sqcup \cdots \sqcup P_n
\]
with suitable identifications of sides.
\end{proof}

Notice that we proved something more subtle. This is, that every triangulation of $\mathcal{S}$ comes from a triangulation of some $\mathbb{M}_i$. Unfortunately, not all triangulations of $\Sigma$ induce a triangulation of $\mathcal{S}$, as it is required that $c_i \subset \partial \Sigma$ be triangulated with $k*n_i$ edges with $k \geq 3$. We will often ignore this fact in the proof, but it will occasionally arise, particularly when drawing examples.

\section{Computing optimal and perfect functions for stratifolds}\label{sec:main-result}
In this section, we will discuss the existence of algorithms to efficiently compute optimal dmf's. In particular, we will define a special family of stratifolds where we will be able to obtain optimal functions in polynomial time and determine if these are perfect or not. 

Lewiner et al proved in \cite{Optimal} that the problem of finding optimal dmf's for any given simplicial complex of dimension 2 reduces to the collapsibility problem presented in \cite{collapsibility}. Hence, it is a MAX-SNP Hard problem. For the case of 2-stratifolds, which are a proper subset of simplical complexes of dimension 2, the problem of finding optimal dmf's also happens to be MAX-SNP Hard. This follows from the fact that the simplicial complexes used in \cite{collapsibility} for their proof turn out to be 2-stratifolds. 

While unable to provide an algorithm in polynomial time to find optimal dmf's for the case of general 2-stratifolds, we can actually say a lot about a special type of them, which we will call \textit{twisted stratifolds}. A twisted stratifold is a stratifold \(X(\mathcal{M}, \mathcal{C})\) such that all the weights of the edges of \(G(X)\) are distinct from \(\pm 1\). 

Given a twisted stratifold \(X(\mathcal{M}, \mathcal{C})\), we can classify it in one of four categories: 

\begin{enumerate}
    \item When all the weights of the white vertices of \(G(X)\) are non-negative and there exists a prime \(p\) such that \(p|\sum_{e \in E(i, j)}w(e)\) for all pairs \((i, j)\) with \(i \leq |\mathcal{M}|\) and \(j \leq |\mathcal{C}|\).

    \item When all the weights of the white vertices of \(G(X)\) are non-negative but there does \textbf{not} exist a prime \(p\) such that \(p|\sum_{e \in E(i, j)}w(e)\) for all pairs \((i, j)\) with \(i \leq |\mathcal{M}|\) and \(j \leq |\mathcal{C}|\). 

    \item When there is at least one white vertex in \(G(X)\) with negative weight and \(2|\sum_{e \in E(i, j)}w(e)\) for all pairs \((i, j)\) with \(i \leq |\mathcal{M}|\) and \(j \leq |\mathcal{C}|\). 

    \item When there is at least one white vertex in \(G(X)\) with negative weight and \(2\nmid \sum_{e \in E(i, j)}w(e)\) for some pair \((i, j)\) with \(i \leq |\mathcal{M}|\) and \(j \leq |\mathcal{C}|\). 
\end{enumerate}

As we will see in this section, we can always construct optimal dmf's for triangulations of twisted stratifolds in polynomial time. Furthermore, we are able to determine if they admit perfect dmf's or not. Before presenting the algorithm that will be used, we need to introduce some terminology. 

Let \(K\) be a triangulation of \(X(\mathcal{M}, \mathcal{S})\). We will call all 0-cells of \(K\) \textbf{vertices}. If a vertex lies on \(\rho(\partial \mathbb{M}_i)\), where \(\mathbb{M}_i \in \mathcal{M}\) and \(\rho\) is the projection from Theorem \ref{Simplicial Decomposition}, we call it a \textbf{boundary vertex}. If it does \textit{not}, we call it an \textbf{internal vertex}.

On the other hand, we call 1-cells \textbf{edges} and classify them as follows: if it lies on \(\rho(\partial \Sigma)\), we call it a \textbf{boundary edge}; if it connects two vertices in the boundary but does \textit{not} lie on the polygonal boundary, we call it a \textbf{crossing edge}; if it connects an internal vertex with a boundary vertex, we call it a \textbf{bridge edge}; and if it connects two internal vertices, we call it an \textbf{internal edge}. Finally, we call all 2-cells \textbf{faces}. 

Recalling the description of the simplicial decompositions of our spaces given by Theorem \ref{Simplicial Decomposition}, we know that these are just collections of triangulated polygons with some of their sides identified. Moreover, we can subdivide these polygons into smaller polygons that have no crossing edges, as shown in Figure \ref{fig:algorithm_1} (a). 

With this in mind, we present the following algorithm to obtain a dmf for any triangulation \(K\) of stratifold \(X(\mathcal{M}, \mathcal{C})\) in linear time. This algorithm is a slight modification of the one presented in \cite{Optimal}, following the heuristic given by the authors for the case of non-regular edges and performing the steps in a different order. 

\begin{figure}[h]
    \centering
    \includegraphics[scale = 0.85]{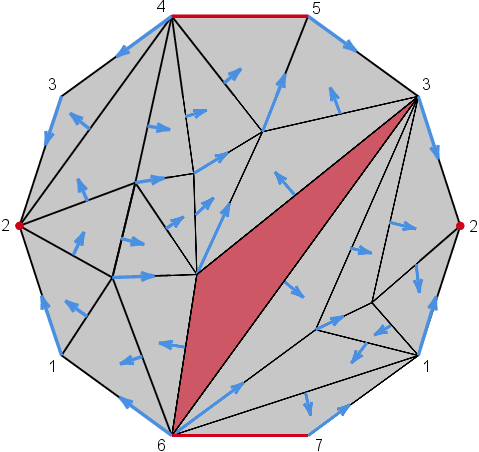}
    \caption{Result of applying the algorithm to a triangulation of the space \(S(\mathbb{S}^2, (1, 1))\).}
    \label{fig:algotihm_result}
\end{figure}

\begin{enumerate}
    \item For each subcomplex \(K_i\) of \(K\) corresponding to a surface \(\mathbb{M}_i \in \mathcal{M}\), do the following: 
    \begin{enumerate}[label*=\arabic*.]
    \item For each polygon \(P_{i, 1}, P_{i,2}, \dots, P_{i, r_i}\) of \(K_i\) defined by the crossing edges, perform the following steps:
    \begin{enumerate}[label*=\arabic*.]
        \item Consider the graph consisting of all the internal vertices and edges of the polygon, along with a boundary vertex and a bridge edge that connects it with the graph.
        \item Compute a spanning tree of the graph obtained in the previous step. 
        \item Define a discrete vector field \(V_{i, j}\) associated with a dmf on the spanning tree, following the construction from the proof of Theorem \ref{Perfect in Graphs}, such that its unique critical 0-cell is the boundary vertex.
        \item Construct a new graph where each 2-cell of our space corresponds to a vertex, and connect two vertices if and only if the 2-cells associated to them share a non-boundary edge that was not used in step 1.1.2. Each of these graphs is a tree, as it will be shown in Theorem \ref{dual_is_tree}. 
    \end{enumerate}
    \item Connect all the trees obtained in step 1.1.4 by adding the duals of the crossing edges. This also gives us a tree \(T\) since we are connecting \(r\) trees with \(r - 1\) edges. 

    \item Construct the set \(V_{T_i}\) as follows: choose an arbitrary 2-cell \(F_{i, 0}\) of \(K_i\) and consider its associated vertex \(v_{i, 0} := v(F_{i, 0})\) in \(T\). For each edge \(e\) in \(T\), add \((e, F(e))\) to \(V_F\), where \(F(e)\) is the unique 2-cell that corresponds with the vertex \(v(F(e))\) of \(e\) farthest from \(v_{i, 0}\) in \(T\). Note that for \(e \neq e'\) we have that \(F(e) \neq F(e')\), since there are no loops in \(T\).       
    \end{enumerate}
    
    \item Consider the graph formed solely by the boundary vertices and edges, accounting for all the identifications. Construct a discrete vector field \(V_B\) associated with a dmf in this graph, following again the proof of Theorem \ref{Perfect in Graphs}.  
    
    \item Take \[V := V_B \cup \left(\bigcup_{i = 1}^{|\mathcal{M}|}V_{T_i}\right) \cup \left(\bigcup_{i = 1}^{|\mathcal{M}|}\bigcup_{j = 1}^{r_i} V_{i, j}\right).\]
\end{enumerate}

The result of applying the algorithm in the subcomplex \(K_i\) corresponding to a given surface \(\mathbb{M}_i\) is shown in Figure \ref{fig:algotihm_result}. Moreover, Figure \ref{fig:algorithm_1} shows the detailed process.

Since a spanning tree of a graph can be found in linear time with respect to the number of edges the graph has, our algorithm runs in \(O(|K|)\) time. 

We will now show the correctness of the algorithm.

\begin{theorem}
    Let \(K\) be a triangulation of \(\mathbb{D}^2\), and \(T\) a tree in the 1-skeleton of \(K\) with exactly one vertex on the boundary of \(K\). Then \(K \setminus T\) is connected.
\end{theorem}
\begin{proof}
    We will proceed by induction on the number of edges in the tree \(T\). If \(T\) has 0 edges, then we are just removing a vertex over the boundary of \(K\), which gives us a connected space. 
    
    Assume that the statement holds for any tree with \(n\) edges. Let \(T\) be a tree of \(n + 1\) edges with exactly one vertex \(v\) on the boundary of \(K\). Since \(T\) is a tree with at least one edge, we can consider a leaf \(w \neq v\). By assumption, \(w\) will be in the interior of \(K\). Let \(e\) be the unique edge incident to \(w\), and denote \(T' := T \setminus \{e\}\). Then, \(T'\) is a tree with \(n\) edges and still has exactly one vertex over \(K\), which is \(v\). By the inductive hypothesis, this implies that \(K \setminus T'\) is connected.

    Now, note that \(e\) has exactly one endpoint over \(T'\), as the other one is \(w\), which lies in the interior of \(K \setminus T'\). Hence, removing it from \(K \setminus T'\) dos not disconnect the space, implying that \((K \setminus T') \setminus \{e\} = K \setminus T\) is a connected space, as desired.
\end{proof}

\begin{corollary}\label{tree_connected}
The graphs constructed in step 1.1.4 are connected. 
\end{corollary}

\newpage
\begin{figure}[h!]
    \centering
    \includegraphics[scale = 0.68]{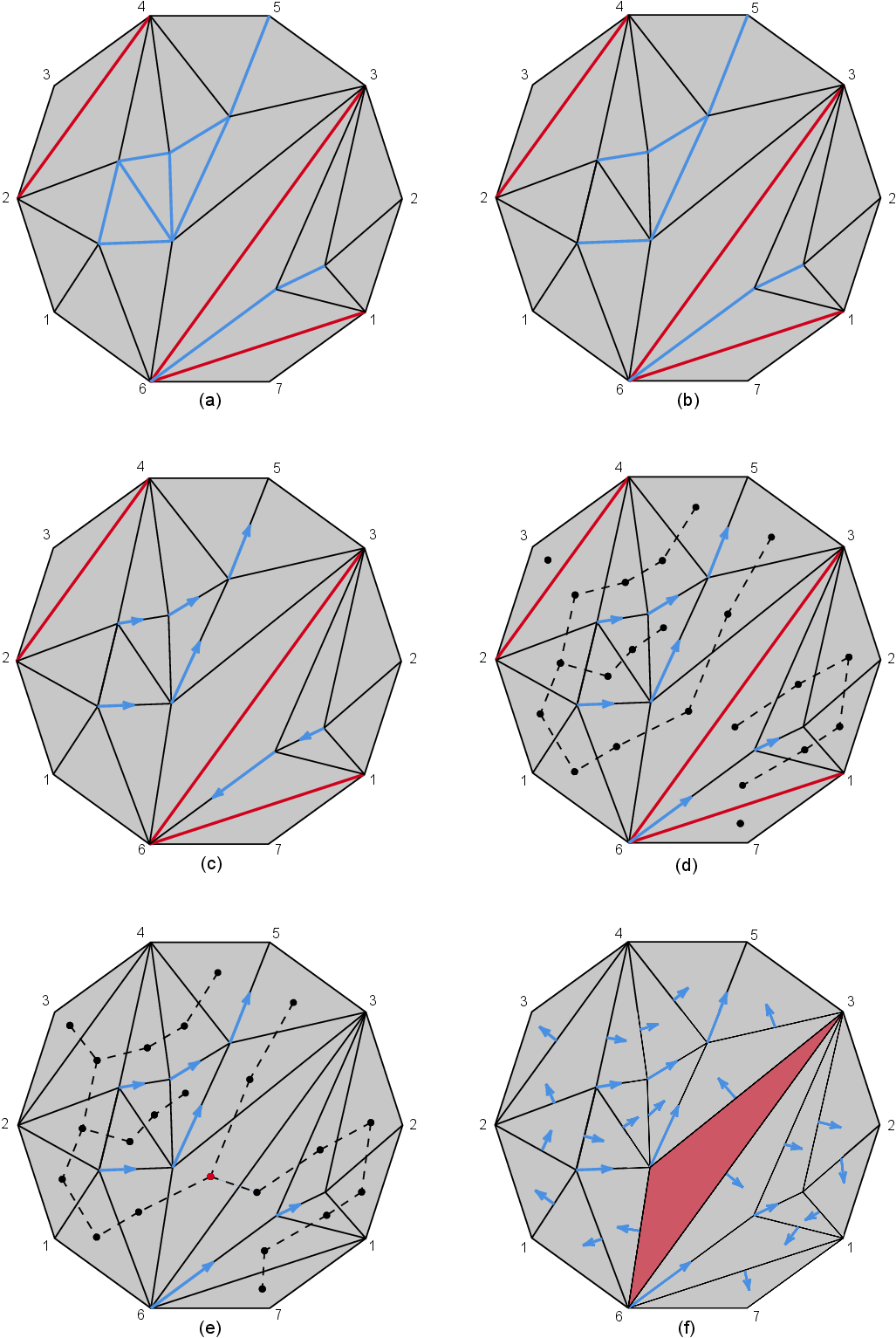}
    \caption{(a) Step 1.1.1.1 of the algorithm. Red edges are the crossing edges and blue edges are the edges selected in this step. (b) Step 1.1.2 (c) Step 1.1.3 (d) Step 1.1.4 (e) Step 1.2 (f) Step 1.3}
    \label{fig:algorithm_1}
\end{figure}

\begin{theorem}\label{dual_is_tree}
    Let \(G_j\) be the graph constructed in step 1.1.4 of the Algorithm for polygon \(P_{i, j}\). Then \(G_j\) is a tree. 
\end{theorem}
\begin{proof}
    By Corollary \ref{tree_connected}, we know that \(G_j\) is a connected graph, which implies that \(\beta_0(G_j) = 1\). Now, let \(V_j\) and \(E_j\) be the number of vertices and edges of \(G_j\), respectively. By Euler's characteristic formula, we know that 
    \begin{equation}\label{euler_dual_is_tree}
        1 - \beta_1(G_j) = V_i - E_i.
    \end{equation}
    Now, denote the total number of 0-cells (vertices), 1-cells (edges), and 2-cells (faces) of our polygon \(P_{i, j}\) as \(V, E\) and \(F\), respectively. Again, by Euler's characteristic formula, we know that \(1 = V - E + F\). Since the number of boundary vertices \(V_\partial\) is equal to the number of boundary edges \(E_\partial\) before identifications, this equality implies that \(1 = (V - V_\partial) - (E - E_\partial) + F\). Moreover, if we denote by \(V_I\) the number of internal vertices, then \(V - V_\partial = V_I\). Since \(F = V_i\), 
    \[1 = V_I - (E - E_\partial) + V_i = V_i - (E - E_\partial - V_I).\]
    Finally, note that \(E_i = E - E_\partial - V_I\), as there is exactly one edge in \(G_i\) for each edge of \(X\) that is not a boundary edge and that was not chosen to be part of the tree, which has \(V_I\) edges. This implies that \(V_i - E_i = 1\). Combining this equality with \ref{euler_dual_is_tree}, we conclude that \(\beta_1(G_i) = 0\). Hence, \(G_i\) is a tree. 
\end{proof}

\begin{theorem}\label{Algorthm Critical Cells}
    The set \(V\) obtained by the former algorithm is a discrete vector field. Moreover, its induced dmf has one critical cell of dimension 0, and \(n\) critical cells of dimension 2. 
\end{theorem}
\begin{proof}
    First, we will show that \(V\) is a discrete vector field. Let \(v\) be a vertex of our triangulation. If \(v\) is an internal vertex, in step 1.1.3 we paired it with a unique internal edge and with no other type of edges. On the other hand, if \(v\) is a boundary vertex, we either paired it with a unique boundary edge in step 2 or left it unpaired. Hence, every vertex of our triangulation appears in at most one pair of \(V\). 

    Consider now an edge \(e\) of our triangulation. If it is a boundary edge, we either paired it with a unique boundary vertex in step 2 or left it unpaired. If it is a crossing edge, it is paired with a unique face in step 1.3. On the other hand, if it is a bridge edge or an internal edge, we have three possibilities: firstly, we could have paired it in step 1.1.3 with a unique vertex. Secondly, we could have paired it with a unique face in step 1.3. In this case \(e\) can't be paired with a vertex, as we excluded those explicitly in step 1.1.4 when constructing the tree \(T\). Lastly, we could have left it unpaired. Therefore, every edge of our triangulation appears in at most one pair of \(V\).

    Finally, note that all faces \(F\) of our triangulation were paired with unique edges in step 1.3. Hence, \(V\) is a discrete vector field.  

    We will continue proving that there are no closed \(V\)-paths. For this, consider first a \(V\)-path of the form
    \[v_0 < e_0 > v_1 < e_1 > \dots < e_{r - 1} > v_r,\]
    where the \(v_i\)'s are vertices and the \(e_i\)'s are edges.
    
    If \(v_0\) is a boundary vertex, by construction this \(V\)-path is actually a \(V_B\)-path (boundary vertices/edges are only paired with boundary edges/vertices), and cannot be closed, as \(V_B\) is a gradient discrete vector field. 
    
    On the other hand, if \(v_0\) is an internal vertex of some \(K_i\), our path can either be entirely contained in the unique gradient vector field \(V_{i, j}\) constructed in step 1.1.3 corresponding to \(v_0\) or not. In the first case it can't be a closed path as it is a \(V_{i, j}\)-path. In the second case, for a path of this type to not be a \(V_i\)-path, it must have reached the critical vertex of \(V_i\), which is a boundary vertex, and from where you can only continue a path using boundary edges and vertices, being impossible to return to \(v_0\). Hence, in this case we also can't have a closed \(V\)-path. This lets us conclude that there are no closed \(V\)-paths starting in vertices.

    Now, consider a \(V\)-path of the form
    \[e_0 < F_0 > e_1 < F_1 > \dots < F_{r - 1} > e_r,\]
    where the \(e_i\)'s are edges and the \(F_i\)'s are faces. 

    Note that, by construction, this path should be contained in some \(V_{T_i}\). Since \(V_{T_i}\) is a discrete vector field and it was constructed over a tree, it is actually a gradient vector field \cite{texbook}. Hence, our path can't be a closed path, which implies that there are no closed \(V\)-paths starting with an edge. 

    Since our triangulations have cells of dimension at most 2, the former discussion shows that there are no closed \(V\)-paths. In other words, \(V\) is a gradient vector field. 
    
    Finally, we will compute the number of critical cells of each dimension. 
    
    \textit{Number of critical 0-cells.} Note that in step 1.1.3 we paired all the internal vertices. Moreover, since the graph from step 2 is connected (as our space was originally connected), according to the algorithm presented in Theorem \ref{Perfect in Graphs}, in this step we paired all the boundary vertices except for the one we chose as root. Hence, we have exactly one 0-cell.

    \textit{Number of critical 2-cells.} In step 1.3 we paired all 2-cells of \(K_i\) except for the one we took as the root of our tree. Hence, we have exactly \(|\mathcal{M}|\) critical 2-cells.
\end{proof}

With this, we can proceed to prove that the algorithm always gives an optimal dmf. 

\begin{theorem}\label{Perfect 1 and 3}
    The dmf \(f\) that the algorithm produces is perfect for all traingulations of twisted stratifolds of type 1 or 3. 
\end{theorem}
\begin{proof}
    We will only prove the Theorem for spaces of type 1, as the proof is analogous for spaces of type 3. 
    
    Let \(X(\mathcal{M}, \mathcal{C})\) be a twisted stratifold of type 1. Since stratifolds are connected, \(\beta_0(X; \Z_p) = 1 = m_0(f)\). On the other hand, by Theorem \ref{Homology Orientable} we know that \(\beta_2(X; \Z_p) = |\mathcal{M}| = m_2(f)\). Finally, by the Euler characteristic and Theorem \ref{Morse Inequalities}, we have that \[\beta_1(X; \Z_p) = \chi(X) - \beta_0(X; \Z_p) - \beta_2(X; \Z_p) = \chi(X) - 1 - n = \chi(X) - m_0(f) - m_2(f) = m_1(f).\]
    Therefore, \(f\) is perfect. 
\end{proof}

\begin{theorem}\label{2 cells}
    Let \(X(\mathcal{M}, \mathcal{C})\) be a twisted stratifold and \(K\) be a triangulation of \(X\). If \(f : K \to \R\) is a dmf, then \(m_2(f) \geq |\mathcal{M}|\).
\end{theorem}
\begin{proof}
    For the sake of contradiction, assume that there exists a dmf \(f : K \to \R\) such that \(m_2(f) < |\mathcal{M}|\). Consider the Morse matching associated with \(V := - \nabla f\). Since \(m_2(f) < |\mathcal{M}|\) there must exist a subcomplex \(K_i\) of \(K\) corresponding to a surface \(\mathbb{M}_i \in \mathcal{M}\) such that \(K_i\) has no critical 2-cells. Hence, given \(\sigma^2_0 \in K_i\), there exists \(\tau_0^1 \in K_i\) such that \((\tau_0^1, \sigma_0^2) \in V\). Moreover, since \(X\) is a twisted stratifold, all the attaching maps have degree with absolute value greater than \(1\), which implies that there exists at least one \(\sigma_1^2 \in K_i\) with \(\sigma_1^2 \neq \sigma_0^2\) such that \(\tau_0^1 < \sigma_1^2\). 

    Note that we can repeat the former process indefinitely. As the number of 2-cells in \(K_i\) is finite, we can take \(m > 0\) such that \(\sigma_0^2 = \sigma_m^2\), giving us the closed \(V\)-path
    
    \[\tau_0^1 < \sigma_0^2 > \tau_m^1 < \sigma_m^2 > \tau_{m - 1}^1 < \sigma_{m - 1}^2 > \dots < \sigma_1^2 > \tau_0^1,\]
    
    which contradicts Theorem \ref{closed paths}. Therefore, \(m_2(f) \geq |\mathcal{M}|\). 
\end{proof}

From Theorems \ref{Homology Orientable} and \ref{Homology Non Orientable}, we also know that \(\beta_2(X; \mathbb{Z}) < n\) and \(\beta_2(X; \mathbb{F}) < n\) for any field \(\mathbb{F}\). Hence, we obtain the following corollary. 

\begin{corollary}\label{no_perfect}
    There are no perfect dmf's for spaces of Type 2 or 4. 
\end{corollary}

\begin{theorem}\label{Optimal type 2-4}
    The dmf \(f : K \to \R\) obtained by the Algorithm is optimal for spaces of Type 2 and 4.  
\end{theorem}
\begin{proof}
    Since \(X\) is connected, for all dmf \(h : K \to \mathbb{R}\) we have that \(m_0(h) \geq 1 = m_0(f)\). Moreover, by Theorem \ref{2 cells}, \(m_2(h) \geq n = m_2(f)\). Hence, 
    \[\chi(X) = m_2(h) - m_1(h) + m_0(h) \geq n + 1 - m_1(h),\]
    where the first equality follows from \ref{Morse Inequalities}. Moreover we also have that 
    \[\chi(X) = m_2(f) - m_1(f) + m_0(f) = n + 1 - m_1(f).\]
    From these two expressions, it follows that \(m_1(h) \geq m_1(f)\). Since \(h\) was arbitrary, we conclude that \(f\) is optimal. 
\end{proof}

From Theorems \ref{Perfect 1 and 3}, \ref{Optimal type 2-4} and Corollary \ref{no_perfect}, we obtain the following result. 

\begin{theorem}\label{Main theorem}
    Let \(X(\mathcal{M}, \mathcal{C})\) be a twisted stratifold and \(K\) a finite triangulation of \(X\). Then, we can construct an optimal dmf on \(K\) in linear time. Moreover, this function is perfect if and only if the stratifold is of type 1 or 3; in cases 2 and 4, no perfect function exists. 
\end{theorem}

\section{Bibliography}
\renewcommand{\section}[2]{}
\bibliography{sn-bibliography}      

\end{document}